\documentclass[reqno,11pt]{amsart}

\usepackage{graphicx}

\usepackage{epsfig}

\long\def\onefigure#1#2{
\begin{figure*}[tbp]
\begin{center}
#1
\end{center}
\caption{#2}
\end{figure*}
} 

\newcommand{\lipefig}[2]  
{\onefigure{\mbox{\psfig{file=#1.eps}}}{\label{f:#1} #2} }

\usepackage{amssymb}
\usepackage[all,cmtip]{xy}
\usepackage{xcolor}

\newcommand{\R}{\mathbb{R}}
\newcommand{\p}{\mathcal{P}}
\newcommand{\Sp}{\mathbb{S}}
\newcommand{\I}{\mathbf{1}}

\DeclareMathOperator{\interior}{int}
\DeclareMathOperator{\nor}{\mathcal{N}}
\DeclareMathOperator{\card}{card}

\newcommand{\Sn}{\mathbb{S}^{n-1}}
\newcommand{\Hc}{\mathcal{H}}
\newcommand{\D}{{\rm d}}
\newcommand{\dif}{{\rm D}}
\newcommand{\bd}{{\rm bd}\,}

\newcommand{\DKB}{\mathcal{D}_{K,B}}
\newcommand{\nbk}{\mathcal{N}(K,B)}

\newtheorem{Proposition}{Proposition}[section]

\newtheorem{Lemma}[Proposition]{Lemma}
\newtheorem{Theorem}[Proposition]{Theorem}

\numberwithin{equation}{section}

\title{Affine diameters of convex bodies}

\author{Imre B\'ar\'any}
\author{Daniel Hug}
\author{Rolf Schneider}

\email{barany.imre@renyi.mta.hu}
\email{daniel.hug@kit.edu}
\email{rolf.schneider@math.uni-freiburg.de}

\address{R\'enyi Institute of Mathematics, 
Hungarian Academy of Sciences, 
PO Box 127, 1364 Budapest, Hungary
}
\address{Karlsruhe Institute of Technology (KIT),  
Department of Mathematics, 
D-76128 Karlsruhe, Germany
}
\address{Mathematisches Institut, Albert-Ludwigs-Universit\"at, 
D-79104 Freiburg i. Br., Germany}

\thanks{Partially supported by ERC Advanced Research Grant no 267165 (DISCONV). The first
author was supported by Hungarian National Foundation Grant K 83767. 
The second author was partially supported by the German Research Foundation (DFG) under the grant HU 1874/4-2.}


\date{}

\dedicatory{}

\commby{}

\begin{document}

\begin{abstract}
We prove sharp inequalities for the average number of affine diameters through the points of a convex body $K$ in ${\mathbb R}^n$. These inequalities hold if $K$ is either a polytope or of dimension two. An example shows that the proof given in the latter case does not extend to higher dimensions.
\end{abstract}

\maketitle

\section{Introduction}

An {\em affine diameter} of an $n$-dimensional convex body in ${\mathbb R}^n$ is a closed segment connecting two points that lie in distinct parallel supporting hyperplanes of the body. Much work has been done on intersection properties of affine diameters and on the characterization of special convex bodies, such as simplices, by such intersection properties. We refer the reader to the survey article by Soltan \cite{Sol05}. As the author points out, that survey does not cover results on affine diameters of typical convex bodies, in the Baire category sense; it also does not touch average numbers of intersections. A result of Baire type, proved by B\'{a}r\'{a}ny and Zamfirescu \cite{BZ90}, says that {\em in most convex bodies, most points belong to infinitely many affine diameters}. This, however, does not imply that the average number of affine diameters through the points of a typical convex body must be infinite, since the set of most points addressed in the theorem can be of measure zero. In fact, in the plane it follows from a result of Hammer and Sobczyk \cite{HS53} that for a convex body with no pair of boundary segments in distinct parallel supporting lines, the set of points through which there pass infinitely many affine diameters is of measure zero.

In this paper, we are concerned with the average number of affine diameters through the points of a convex body. For a convex body $K\subset {\mathbb R}^n$ and a point $z\in{\rm int}\,K$ we denote by $N_a(K,z)$ the number ($\infty$ admitted) of affine diameters passing through $z$. We define the mean number of affine diameters passing through a point of $K$ by
\begin{equation}\label{D1}
N_a(K):=\frac{1}{V_n(K)}\int_K N_a(K,z)\, \D z,
\end{equation}
where $V_n$ denotes the volume and $\D z$ indicates integration with respect to Lebesgue measure. (The function $N_a(K,\cdot)$ is Borel measurable; see Section 2.) Some caution is advisable. Recall that for a convex body $K\subset {\mathbb R}^n$ and a unit vector $u\in\R^n$, the set $F(K,u)$ is the support set of $K$ with outer normal vector $u$. If there is a vector $u$ such that $\dim (F(K,u)+F(K,-u))=n-1$ and $\dim F(K,u)+\dim F(K,-u)>n-1$, then there is a set of positive measure in $K$ through each point of which there pass infinitely many affine diameters, thus $N_a(K)=\infty$ for such a body. Incidentally, this shows that the function $N_a$ is not continuous on the space of convex bodies with the Hausdorff metric. Generally, we say for a convex body $K\subset{\mathbb R}^n$ that $K$ and $-K$ are {\em in general relative position} if
$$ \dim(F(K,u) + F(K,-u)) = n-1 \;\Rightarrow\; \dim F(K,u) + \dim F(K,-u) = n-1$$
for all $u\in\Sn$. (We warn the reader that this notion appears in the literature also with a more restrictive definition.) Thus, $N_a(K)<\infty$ can only be expected if $K$ and $-K$ are in general relative position. The following seems to be unknown.

\vspace{3mm}

\noindent{\bf Question.} Is $N_a(K)<\infty$ if $K$ and $-K$ are in general relative position?

\vspace{3mm}

The following theorems give affirmative answers, in a strengthened form, if either $K$ is a polytope or if the dimension is two.

\begin{Theorem}\label{Thm1}
Let $P\subset\R^n$ be an $n$-polytope such that $P$ and $-P$ are in general relative position. Then
\begin{equation}\label{T1}
N_a(P)=\frac{n+1}{V_n(P)}\int_0^1V_n((1-t)P-tP)\, \D t -1.
\end{equation}
This implies that
\begin{equation}\label{T2}
n< N_a(P)\le 2^n-1.
\end{equation}
Equality on the right-hand side holds if and only if $P$ is a simplex. The lower bound $n$ is sharp, but is not attained.
\end{Theorem}

\begin{Theorem}\label{Thm2}
Let $K\subset\R^2$ be a two-dimensional convex body such that $K$ and $-K$ are in general relative position. Then
\begin{equation}\label{T3}
1\le N_a(K) \le \frac{V_2(K-K)}{2V_2(K)}\le 3.
\end{equation}
Equality on the left side is attained if and only if $K$ is centrally symmetric. Equality on the right side is attained if and only if $K$ is a triangle.
\end{Theorem}

A comparison of the sharp lower bounds in (\ref{T2}) and (\ref{T3}) shows incidentally that the function $N_a$ is not continuous, even if restricted to the set of planar convex bodies $K$ for which $K$ and $-K$ are in general relative position. In fact, such a body $K\in{\mathcal K}^2$ which is centrally symmetric and hence satisfies $N_a(K)=1$, can be approximated arbitrarily closely by convex polygons $P$ with $P$ and $-P$ in general relative position, for which $N_a(P)>2$.

To the question posed above, we can also give a positive answer in $n$ dimensions, if we assume in addition that $K$ has a support function of class $C^2$. Since this is technically more involved, it will be considered elsewhere.

\section{Preliminaries}

Let $\R^n$ be equipped with the standard scalar product $\langle\cdot,\cdot \rangle$ and the induced norm $\|\cdot\|$. We write $o$ for the origin (zero vector) of $\R^n$. The linear subspace orthogonal to a vector $u\neq o$ is denoted by $u^\perp$. Unit ball and unit sphere of ${\mathbb R}^n$ are denoted, respectively, by $B^n$ and  $\Sn$.

We denote by ${\mathcal K}^n$ the set of convex bodies (nonempty, compact, convex subsets) in $\R^n$. The set ${\mathcal P}^n\subset{\mathcal K}^n$ is the set of convex polytopes. For a polytope $P$, the set of $r$-dimensional faces of $P$ is denoted by ${\mathcal F}_r(P)$, $r=0,\dots,n-1$.

The support function $h(K,\cdot)$ of a convex body $K$ is defined by $h(K,x):= \max\{\langle x,y\rangle: y\in K\}$, and for $u\in\Sn$, the  hyperplane
$$H(K,u):= \{x\in\R^n: \langle u,x\rangle = h(K,u)\}$$
is the supporting hyperplane of $K$ with outer normal vector $u$. The face (or support set) of $K$ in direction $u$ is the set $F(K,u)= K\cap H(K,u)$. By $N(K,x)$ we denote the normal cone of $K$ at its boundary point $x$, that is, the set of all outer normal vectors to $K$ at $x$, together with the zero vector. If $K$ is smooth (i.e., has only regular boundary points), then to each $x\in{\rm bd}\,K$ there is a unique outer unit normal vector to $K$ at $x$; we denote it by $u_K(x)$.

As usual, $K+M=\{x+y:x\in K,\, y\in M\}$ for $K,M\in{\mathcal K}^n$ and $\mu K:= \{\mu x: x\in K\}$ for $\mu\in \R$. In particular,
$$ {\rm D}K:= K-K =\{x-y:x,y \in K\}$$
is the difference body of $K$. For $u\in \Sn$, $h({\rm D}K,u)$ is the width of $K$ in direction $u$.

Lebesgue measure on ${\mathbb R}^n$ is denoted by $\lambda_n$. We also make use of the $k$-dimensional Hausdorff measure,  ${\mathcal H}^k$. For the volume of convex bodies in $\R^n$ we prefer the notation $V_n$, and by $V(\cdot,\dots,\cdot)$ ($n$ arguments) we denote the mixed volume. For this, and for some notation and results used below, we refer to \cite{Sch14}, Section 5.1. The $(n-1)$-dimensional mixed volume of convex bodies lying in parallel $(n-1)$-dimensional affine subspaces is denoted by $v(\cdot,\dots,\cdot)$ ($n-1$ arguments).

To show that the integral (\ref{D1}) is defined, let $M_{k,m}$ be the set of all points $x\in{\rm int}\,K$ through which there pass at least $k$ affine diameters, each two of which form an angle at least $1/m$, where $k,m\in{\mathbb N}$. If $(x_j)_{j\in{\mathbb N}}$ is a sequence in $M_{k,m}$ converging to some point $x\in{\rm int}\,K$, then, choosing suitable convergent subsequences of affine diameters, we see that through $x$ there pass at least $k$ affine diameters, each two of them forming an angle at least $1/m$. Thus, the set $\{x\in{\rm int}\,K: N_a(K,x)\ge k\}$ is the union of countably many closed sets and hence is a Borel set. Since this holds for all $k\in{\mathbb N}$, the function $N_a(K,\cdot)$ is Borel measurable.

\section{Proof of Theorem \ref{Thm1}}

Let $P\in\p^n$ be an $n$-polytope with the property that $P$ and $-P$ are in general relative position. By $D$ we denote the set of all points $z$ which lie in the convex hull of any two faces of $P$ where this convex hull has dimension less than $n$. Let $z\in P\setminus D$ and $z\in[x,y]$, where $[x,y]$ is an affine diameter of $P$. Then there is some $u\in \Sn$ such that $x\in F(P,u)$ and $y\in F(P,-u)= -F(-P,u)$, and hence $x-y\in F(P,u)+F(-P,u)=F(P-P,u)=F({\rm D}P,u)$ (where we used \cite{Sch14}, Thm. 1.7.5(c)). Let $r:=\dim F(P,u)$ and $s:= \dim F(P,-u)$. Since $z\notin D$, we have $\dim(F(P,u)+F(P,-u))=n-1$. Since $P$ and $-P$ are in general relative position, it follows that $r+s=n-1$. Hence, every affine diameter of $P$ through $z$ is of the form $[x,y]$ with $x\in F$ and $y\in -G$ for some faces $F\in\mathcal{F}_r(P)$ and $G\in\mathcal{F}_s(-P)$ satisfying $r+s= n-1$ and $F+G\in\mathcal{F}_{n-1}({\rm D}P)$. Since $z\notin D$, in fact  $x\in {\rm relint}\,F$  and $y\in {\rm relint}(-G)$. For any such pair $F,G$ we define
$$
A(F,G):=\{(1-t)x-ty:t\in [0,1],\,x\in F,\,y\in G\}.
$$
Then for $z\in P\setminus D$ there is a one-to-one correspondence between the affine diameters through $z$ and the pairs $F,G$ with $z\in A(F,G)$. Thus, for $z\in P\setminus D$,
$$
N_a(P,z)=\frac{1}{2}\sum_{r=0}^{n-1}\sum_{\ast}\I\{z\in A(F,G)\},
$$
where the summation $\sum_{\ast}$ extends over all faces $F\in\mathcal{F}_r(P)$ and all faces $G\in\mathcal{F}_{n-1-r}(-P)$ such that $F+G\in\mathcal{F}_{n-1}({\rm D}P)$. Since $D$ is of measure zero, we conclude that
\begin{equation}\label{0}
N_a(P)=\frac{1}{2V_n(P)}\sum_{\ast}\lambda_n(A(F,G)).
\end{equation}
We can write this as
\begin{equation}\label{1}
N_a(P)=\frac{1}{2V_n(P)}\sum_{u\in\Sn}\lambda_n(A(F(P,u),F(-P,u))),
\end{equation}
since a summand is different from zero only if $F(P,u)+F(-P,u)=F({\rm D}P,u)$ is an $(n-1)$-face of ${\rm D}P$.

Assume that $F,G$ is such a pair as in (\ref{0}), and let $u$ be the outer unit normal vector of ${\rm D}P$ at $F+G$. The width of $P$ in direction $u$ is given by $h:=h({\rm D}P,u)$. Writing
$$ H_\tau:= (1-\tau/h)H(P,u)+(\tau/h)H(P,-u)$$
for $0\le \tau\le h$, we have
$$ A(F,G)\cap H_\tau =(1-\tau/h)F-(\tau/h)G.$$
Therefore, Fubini's theorem together with the substitution $\tau= th$ gives
\begin{align}
\lambda_n(A(F,G)) &= \int_0^h \lambda_{n-1}(A(F,G)\cap H_\tau)\,\D \tau \nonumber\\
&= \int_0^1 \lambda_{n-1}((1-t)F-tG)h\,\D t\nonumber\\
&= \int_0^1 \sum_{k=0}^{n-1} \binom{n-1}{k} (1-t)^kt^{n-1-k} v(F[k],-G[n-1-k])h\,\D t\nonumber\\
&= \frac{1}{n}\sum_{k=0}^{n-1} h({\rm D}P,u)v(F[k],-G[n-1-k]),\label{2}
\end{align}
where $v(F[k],-G[n-1-k])$ denotes the $(n-1)$-dimensional mixed volume of $F$ taken $k$ times and $-G$ taken $n-1-k$ times.

Combining equations (\ref{1}) and (\ref{2}) and using formula (5.23) of \cite{Sch14}, we get
\begin{align}
N_a(P)&=\frac{1}{2V_n(P)} \sum_{u\in\Sn} \frac{1}{n}\sum_{j=0}^{n-1}h({\rm D}P,u)v(F(P,u)[j],F(-P,u)[n-1-j])\nonumber\\
&=\frac{1}{2V_nP)}\sum_{j=0}^{n-1}V({\rm D}P,P[j],-P[n-1-j])\nonumber\\
&=\frac{1}{2V_n(P)}\sum_{j=0}^{n-1}\big\{V(P[j+1],-P[n-1-j])+V(P[j],-P[n-j])]\big\}\nonumber\\
&=\frac{1}{V_n(P)}\left[\sum_{k=0}^{n}V(P[k],-P[n-k])\right]-1.\label{3}
\end{align}

On the other hand, we have
\begin{align}
&\int_0^1V_n((1-t)P-tP)\, \D t\nonumber\\
&=\int_0^1\sum_{k=0}^{n}\binom{n}{k}(1-t)^k t^{n-k} V(P[k],-P[n-k])\, \D t\nonumber\\
&=\frac{1}{n+1}\sum_{k=0}^{n}V(P[k],-P[n-k]).\label{4}
\end{align}

In view of (\ref{3}) and (\ref{4}), the proof of equation (\ref{T1}) is complete.

The inequalities (\ref{T2}) follow from (\ref{T1}) and the inequalities
\begin{equation}\label{RS}
V_n(K)\le\int_0^1V_n((1-t)K-tK)\, \D t\le \frac{2^n}{n+1}V_n(K),
\end{equation}
which are due to Rogers and Shephard (\cite[Theorem 2]{RS58}; note that the formulation there involves an associated convex body, but is equivalent to (\ref{RS})). They hold for all $n$-dimensional convex bodies $K\in\mathcal{K}^n$. Equality on the right holds if and only if $K$ is an $n$-simplex. Equality on the left holds if and only if
$K$ is centrally symmetric. A polytope $P$ for which $P$ and $-P$ are in general relative position cannot be centrally symmetric,
hence we have strict inequality on the left side of (\ref{T2}). On the other hand, the right side of (\ref{T1}) is a continuous function of $P$ in the Hausdorff metric. A centrally symmetric polytope can be approximated arbitrarily closely by polytopes $P$ with $P,-P$ in general relative position (as follows, e.g., from the proof of \cite[Theorem 3.7]{Sch94}). Therefore, the lower bound $n$ in (\ref{T2}) cannot be replaced by a larger one.
\qed

\vspace{3mm}

\noindent{\bf Remark.} In the planar case, formula (\ref{T1}) can be written as
\[
N_a(P)=\frac {V_2(P-P)}{2V_2(P)}.
\]
This can be deduced directly from (\ref{0}) as follows. The pair $(F,G)$ with $\dim F=1$, $\dim G=0$ takes part in the sum (\ref{0}) exactly when $F$ is an edge of $P$ and $-G=\{v(F)\}$ is the unique vertex opposite to $F$. Then $A(F,G)$ is the triangle ${\rm conv} (F \cup \{v(F)\})$, which is a translate of the triangle ${\rm conv}((F-v(F))\cup \{o\})$. These triangles, together with their reflections about the origin, are easily seen to form a triangulation of $P-P$. The sum of the areas of these triangles is then indeed half the area of $P-P$.

\section{Relative normals}\label{relnormal}

For treating affine diameters, we first develop some methods and results for relative normals. We do this in a slightly more general fashion than needed for the affine diameters, since it requires little additional effort and is of independent interest.

We assume that a fixed convex body $B\in{\mathcal K}^2$ with $o\in{\rm int}\,B$ is given; we call it the {\em gauge body}. For a nonempty compact set $K\subset {\mathbb R}^2$, the $B$-{\em distance of $x$ from $K$} is defined by
\begin{eqnarray*}
d(K,B,x) &=& \min\{r\ge 0: x\in K+rB\}\\
&=& \min\{r\ge 0: (-rB+x)\cap K\not=\emptyset\}.
\end{eqnarray*}
It is easy to see that $d(K,B,\cdot)$ is a convex function.

Now let $K\in{\mathcal K}^2$ be a convex body. We say that $K$ and $B$ are in {\em general relative position} if ${\rm dim}\,F(K,u)+{\rm dim}\,F(B,u)\le 1$ 
for all $u\in\mathbb{S}^1$. This is the case if and only if ${\rm dim}\,F(K+B,u)={\rm dim}\,F(K,u)+{\rm dim}\,F(B,u)$ for all $u\in\mathbb{S}^1$. Note that if this definition is applied to $K$ and $B=-K$, it is consistent with the definition given in the introduction. 
 In the rest of this section, $K$ and $B$ are fixed convex bodies which are in general relative position.

Let $x\in{\mathbb R}^2\setminus K$. Since $K$ and $B$ are in general relative position, there are a unique point $p(K,B,x)\in{\rm bd}\,K$ and a unique vector $u(K,B,x)\in{\rm bd}\,B$ such that
$$ x=p(K,B,x)+d(K,B,x)u(K,B,x).$$
We call $p(K,B,x)$ the $B$-{\em projection} of $x$ to $K$ and the vector $u(K,B,x)$ a $B$-{\em normal} of $K$ at $p(K,B,x)$.

Next we provide some Lipschitz and differentiability properties.

\begin{Lemma}\label{L4.1} The $B$-projection $p(K,B,\cdot):{\mathbb R}^2\setminus K \to{\rm bd}\,K$ is a Lipschitz map.
\end{Lemma}

\begin{proof} Let $x,y\in{\mathbb R}^2\setminus K$. We abbreviate
$$ R(x):= p(K,B,x)+[0,\infty)(x-p(K,B,x))$$
and define the ray $R(y)$ similarly. Note that all points $z\in R(x)$ satisfy $p(K,B,z)= p(K,B,x)$. We may assume that $e:= p(K,B,x)-p(K,B,y)\not= o$, since otherwise relation (\ref{4.2}) is trivial. Then it follows that
$$
R(x)\cap R(y)=\emptyset.
$$
Therefore, we have either
\begin{equation}\label{4.2}
\|x-y\| \ge \|e\|
\end{equation}
or
$$ \|x-y\| \ge \|p(K,B,x)-\bar y\| = \|e\|\sin\alpha,$$
where $\bar y$ denotes the orthogonal projection of $p(K,B,x)$ to the ray $R(y)$ and $\alpha$ is the angle between the vectors $e$ and $y-p(K,B,y)$. This angle satisfies $\alpha \ge \alpha_0>0$, where $\alpha_0$ is the smallest angle that a vector $b\in {\rm bd}\,B$ can form with a supporting line of $B$ at $b$, which is clearly positive. It follows that
\begin{equation}\label{4.3}
\|x-y\| \ge \|e\|\sin\alpha_0.
\end{equation}
By (\ref{4.2}) and (\ref{4.3}), the Lipschitz continuity of $p(K,B,\cdot)$ is established. \end{proof}

\begin{Lemma}\label{Lemma4.2}
Let $x\in \R^2\setminus K$, and set $t:=d(K,B,x)$. Then $d(K,B,\cdot)$ is differentiable at $x$ if and only if $x$ is a regular boundary point of $K+tB$.
\end{Lemma}

\begin{proof} We introduce the following notation, for both directions of the proof. There is a unique vector $b\in{\rm bd}\,B$ such that $x-tb\in {\rm bd}\,K$. Further, since $x\in {\rm bd}\,(K+t{B})$, there is some vector $u\in N(K+tB,x)\cap \Sp^1$ such that $(x+\R\, u)\cap \interior(K+tB)\neq \emptyset$. Note that $d(K,B,y)=t$ for all $y\in \bd(K+tB)$. There exists a nonnegative convex function $f$ such that
$$ \gamma_{v}(s):= x+sv-f(x+sv)u\in {\rm bd}\,(K+tB)$$
if $v\in u^{\perp}\cap \Sp^1$ and $|s|$ is small enough. The convexity of
$f$ implies the existence of the limit
$$ \gamma_{v}^{\prime}(0;1):=\lim_{s\downarrow 0}\frac{\gamma_{v}(s)-\gamma_{v}(0)}{s}=v-f^{\prime}(x;v)u\,.$$

Let us assume now that $d(K,B,\cdot)$ is differentiable at $x$. Since $d(K,B,\cdot)\circ \gamma_{v}(s)=t$, for $|s|$ sufficiently small, we obtain that
\begin{equation}\label{g2}
\dif d(K,B,x)(v-f^{\prime}(x;v)u)=0
\end{equation}
for $v\in u^{\perp}\cap \Sp^1$. Suppose that $f^{\prime}(x;v)\neq 0$, and hence $f^{\prime}(x;v)>0$, for some $v\in u^{\perp}\cap \Sp^1$. Then the vectors $a_1:= v-f'(x;v)u$ and $a_2:= -v-f'(x;-v)u$ are linearly independent, and therefore (\ref{g2}) implies that $\dif d(K,B,x)= 0$. This contradicts the fact that the directional derivative of $g:=d(K,B,\cdot)$ satisfies $g'(x;b)=1$. Therefore, $f^{\prime}(x;v)= 0$ for $v\in u^{\perp}$, hence $x$ is a regular boundary point of $K+tB$.

Conversely, assume that $x$ is a regular boundary point of $K+t{B}$. Let $v\in u^{\perp}\cap \Sp^1$. We have $d(K,B,x)= d(K,B,\gamma_v(s))=t$. There is a number $r_0>0$ with $r_0B^2\subset B$, and this implies that 
$$|d(K,B,x)-d(K,B,y)|\le r_0^{-1}\|y-x\|.$$ 
Hence,
for $|s|>0$ sufficiently small, we get
\begin{align*}
&\left|\frac{d(K,B,x+sv)-d(K,B,x)}{s}\right|=\left|\frac{d(K,B,x+sv)-d(K,B,\gamma_v(s))}{s}\right| \\
& \le \frac{1}{r_0}\left\|\frac{x+sv-\gamma_v(s)}{s}\right\|
= \frac{1}{r_0}\left|\frac{f(x+sv)}{s}\right|.
\end{align*}
Since $f$ is differentiable at $x$ and $f(x)=0$, this yields that the partial derivatives of the convex function $d(K,B,\cdot)|_{(x+u^{\perp})}$ at $x$ exist. Also $d(K,B,\cdot)|_{(x+\R\, b)}$ is differentiable at $x$, obviously.
Thus, the convex function
$$ (\alpha_1,\alpha_2)\mapsto d\left(K,B,x+\alpha_1 v+\alpha_2 b\right),$$
$(\alpha_1,\alpha_2)$ sufficiently close to $(0,0)$, has partial derivatives at $o$. Hence, by \cite[Theorem 1.5.8]{Sch14}, it is differentiable at $o$. This implies that $d(K,B,\cdot)$ is differentiable at $x$.
\end{proof}

Since $d(K,B,\cdot)$ is convex and therefore differentiable almost everywhere, almost every point $x\in{\mathbb R}^2\setminus K$ is a regular boundary point of $K+d(K,B,x)B$.

\vspace{3mm}

For $\lambda>0$ we consider the map
$$ h_{\lambda}:\R^2\setminus K\to \R^2\setminus K,\quad y\mapsto p(K,B,y)+\lambda(y-p(K,B,y)).$$
Obviously $(h_{\lambda})^{-1}=h_{\lambda^{-1}}$, and  we have $h_{\lambda}(\bd(K+t{B}))=\bd(K+t\lambda B)$ for $t>0$. By Lemma \ref{L4.1}, $h_{\lambda}$ is a bi-Lipschitz map. But then $\dif h_{\lambda}(y)$ has rank two and $(h_{\lambda})^{-1}$ is differentiable at $h_{\lambda}(y)$ if $h_{\lambda}$ is differentiable at $y\in \R^2\setminus K$ (see the proof of Theorem 3.2 in \cite{Walter}). This implies that $p(K,B,\cdot)$ is differentiable at $y\in \R^2\setminus K$ if and only if $p(K,B,\cdot)$ is differentiable at $\bar{y}$ for any $\bar{y}\in R(y)$. The same is true for $d(K,B,\cdot)$, as follows from  Lemma \ref{Lemma4.2}
and the relation 
\begin{equation}\label{Nor}
N(K+tB,x+tb)=N(K,x)\cap N(B,b)
\end{equation} 
(see \cite[Theorem 2.2.1(a)]{Sch14}).

We define $\DKB$ as the set of all $y\in \R^2\setminus K$ such that $p(K,B,\cdot)$ and $d(K,B,\cdot)$ are differentiable at $y$, and hence  at any point of $R(y)\setminus\{y\}$. Then Lemma \ref{L4.1} and Lemma \ref{Lemma4.2} yield that $\mathcal{H}^2(\R^2\setminus(K\cup \DKB))=0$. Since $d(K,B,\cdot)$ is Lipschitz, the coarea formula yields
\begin{eqnarray*}
0 &=& \int_{\R^2 \setminus(K\cup {\mathcal D}_{K,B})} J_1 d(K,B,x)\,\Hc^2(\D x)\\
&=& \int_0^\infty \Hc^1({\rm bd}(K+tB)\setminus {\mathcal D}_{K,B})\,\D t.
\end{eqnarray*}
Let $t_0>0$ be such that $\Hc^1({\rm bd}(K+t_0B)\setminus {\mathcal D}_{K,B})=0$. Let $t>0$. The bi-Lipschitz map $h_{t/t_0}$ maps
${\rm bd}(K+t_0B)\setminus {\mathcal D}_{K,B}$ onto ${\rm bd}(K+tB)\setminus {\mathcal D}_{K,B}$, hence we conclude that
$\Hc^1({\rm bd}(K+tB)\setminus {\mathcal D}_{K,B})=0$ for all $t>0$. 

Our next aim is to introduce generalized relative curvatures on a generalized normal bundle, partly following \cite{Hug99}. The $n$-dimensional Euclidean case of this notion (that is, with $B$ replaced by $B^n$) is sketched in \cite[Section 2.6]{Sch14}.

Choose $y\in \DKB$ and set $t:=d(K,B,y)$. The differential $\dif u(K,B,y)$ of $u(K,B,\cdot)$ at $y$ exists. Let $u:=u_{K+tB}(y)$ be the unique Euclidean outer unit normal vector of $K+tB$ at $y$, and let $v\in u^{\perp}\cap \Sp^1$. We can choose an injective, continuous mapping $\gamma: (-\varepsilon,\varepsilon)\to\bd(K+tB)$ with the properties that $\gamma(0)=y$ and $\gamma$ is differentiable at $0$ with $\gamma^{\prime}(0)=v$. Then $u(K,B,\cdot)\circ\gamma$ maps $(-\varepsilon,\varepsilon)$ into $\bd B$ and is differentiable at $0$, hence the vector $w:= (u(K,B,\cdot)\circ\gamma)'(0)$ exists. Since, by (\ref{Nor}), $u$ is an outer normal vector of $B$ at $u(K,B,y)$, we have $\langle u,w\rangle = 0$. Moreover, since
$$ v=\lim_{s\downarrow 0} \frac{\gamma(s)-y}{s},$$
there is a decreasing null sequence $(s_i)_{i\in{\mathbb N}}$ with $\langle \gamma(s_i)-y,v\rangle >0$ and such that the points $\gamma(s_i)$ are regular boundary points of $K+tB$. Let $u_i$ be the Euclidean outer unit normal vector of $K+tB$ at $\gamma(s_i)$. Then $u_i$ is also an outer unit normal vector of $B$ at $u(K,B,\gamma(s_i))$, as follows from (\ref{Nor}), applied to $x',b'$ with $x'+tb'=\gamma(s_i)$. Since $\langle u_i,v \rangle\ge 0$ (for sufficiently small $s_i$), it follows that
$$ \langle u(K,B,\gamma(s_i))-u(K,B,y),v\rangle \ge 0$$
and hence that $\langle w,v\rangle \ge 0$. Altogether, we have $\langle w,u\rangle =0$ and $\langle w,v\rangle\ge 0$, hence there exists a number $k(K,B,y)\ge 0$ with
$$ \dif u(K,B,y)(v)=k(K,B,y)v.$$
Thus, we have established the following lemma.

\begin{Lemma}\label{Lemma6.2}
Let $y\in \DKB$, and set $u:=u_{K+d(K,\,B,\,y)B}(y)$. Then there is a number $k(K,B,y)\geq 0$ such that $\dif u(K,B,y)(v)= k(K,B,y)v$ for all $v\in u^{\perp}$.
\end{Lemma}

Now, for $y\in \DKB$ and $0<s<d(K,B,y)$, we have
\[y-su(K,B,y)\in
\bd(K+(d(K,B,y)-s)B)\]
and
\begin{equation}\label{g6.1}
u(K,B,y-su(K,B,y))=u(K,B,y).
\end{equation}
Choosing $v\in u^{\perp}\setminus\{o\}$, where
$$ u:=u_{K+d(K,\,B,\,y)B}(y)=u_{K+(d(K,\,B\,,y)-s)B}(y-su(K,B,y)),$$
we
obtain from (\ref{g6.1}) that
\begin{equation}\label{g6.2}
(1-sk(K,B,y))\dif u(K,B,y-su(K,B,y))(v)=k(K,B,y)v.
\end{equation}
Moreover, we have
\begin{equation}\label{g6.3}
\dif u(K,B,y-su(K,B,y))(v)=k(K,B,y-su(K,B,y))v.
\end{equation}
From (\ref{g6.2}) and (\ref{g6.3}) it follows that $k(K,B,y)<1/s$ and
\begin{equation}\label{g6.4}
k(K,B,y-su(K,B,y))=\frac{k(K,B,y)}{1-sk(K,B,y)}.
\end{equation}
Hence we get
\[0\leq k(K,B,y)\leq d(K,B,y)^{-1}.\]
Using (\ref{g6.4}), we see that
$k(K,B,y)=d(K,B,y)^{-1}$ implies
\[k(K,B,y-su(K,B,y))=d(K,B,y-su(K,B,y))^{-1}.\]
Furthermore,
$k(K,B,y)<d(K,B,y)^{-1}$ yields
\begin{align*}
&\frac{k(K,B,y)}{1-d(K,B,y)k(K,B,y)}\\[1mm]
&=\frac{k(K,B,y-su(K,B,y))}{1-d(K,B,y-su(K,B,y)) k(K,B,y-su(K,B,y))}.
\end{align*}

\vspace{2mm}

Before we summarize the obtained results in the next lemma, we define
$$ \nbk := \{(p(K,B,x),u(K,B,x))\in {\rm bd}\,K\times {\rm bd}\,B :x\in {\rm bd}(K+tB)\},$$
which is independent of $t>0$.
The set $\nbk$ is called the $B$-{\em normal bundle} of $K$. For fixed $t>0$, the mappings
\[F:\nbk\to\bd(K+tB),\quad (x,b)\mapsto x+tb,\]
and
\[F^{-1}:\bd(K+tB)\to \nbk,\quad
y\mapsto(p(K,B,y),u(K,B,y)),\]
are Lipschitz maps which are inverse to each other. In particular, this shows
that $\nbk$ is a closed, $1$-rectifiable subset of $\R^2\times \R^2$.

\begin{Lemma}\label{Lemma6.3}
Let $(x,b)\in \nbk$ be such that $y:=x+tb\in \DKB$ for some (and hence for all) $t>0$. Then
$$ k(K,B;x,b):=\frac{k(K,B,y)}{1-tk(K,B,y)}\in [0,\infty] $$
is independent of the particular choice of $t>0$. Moreover, $k(K,B;x,b)$ is defined for $\mathcal{H}^1$ almost all $(x,b)\in\nbk$, and $k(K,B;x,b)=\infty$ if and only if $k(K,B,y)=1/t$.
\end{Lemma}

Next, we express the (Euclidean) first-order area measures of $K$ and $B$ in terms of generalized curvatures. As a preparation, we describe the tangent space of the $B$-normal bundle in terms of these generalized curvatures.

Let $y\in \DKB\cap \bd(K+tB)$, $t>0$, and $(x,b):=F^{-1}(y)$. Then, if $u:=u(x,b):=u_{K+tB}(y)$ and $v\in {\mathbb S}^1\cap u^{\perp}$, we obtain, recalling that $p(K,B,y)=y-tu(K,B,y)$,
\begin{align}
&\text{Tan}^1\left(\nbk,(x,b)\right)=\dif F^{-1}(y)(u^{\perp})\nonumber\\
&\quad =\text{lin}\left\{\left(\dif p(K,B,y)(v),\dif u(K,B,y)(v)\right)\right\}\nonumber\\
&\quad =\text{lin}\left\{\left((1-tk(K,B,y))v,k(K,B,y)v\right)\right\}\nonumber\\
&\quad
=\text{lin}\left\{\left(\frac{1}{\sqrt{1+k(K,B;x,b)^2}}\,v,\frac{k(K,B;x,b)}
{\sqrt{1+k(K,B;x,b)^2}}\,v\right)\right\}.\label{Tanspace}
\end{align}
Here, we had to distinguish the cases $k(K,B,y)<1/t$ and $k(K,B,y)=1/t$. These facts are used in the proof of the following lemma and in the next section.

We recall that $S_1(K,\cdot)$ denotes the first-order area measure of a convex body (see \cite[Section 4.2]{Sch14}); in particular, in the plane it is the length measure (see also \cite[Subsection 8.3.1]{Sch14}).

\begin{Lemma}\label{Lemma4.5} Let $\omega\subset\Sp^1$ be a Borel set. Then
\begin{align*}
S_1(K,\omega)&=\int_{\nor(K,B)}\I\{u(x,b)\in \omega\}\frac{1}{\sqrt{1+k(K,B;x,b)^2}}\,\mathcal{H}^1(\D(x,b))\\
\intertext{and}
S_1(B,\omega)&=\int_{\nor(K,B)}\I\{u(x,b)\in \omega\}\frac{k(K,B;x,b)}{\sqrt{1+k(K,B;x,b)^2}}\,\mathcal{H}^1(\D(x,b)).
\end{align*}
\end{Lemma}

\begin{proof}
It follows from \eqref{Tanspace} that the approximate Jacobian of the surjective Lip\-schitz map $\Pi_1:\nor(K,B)\to\bd K$ with  $(x,b)\mapsto x$ is equal to
$$
{\rm ap}\, J_1\Pi_1(x,b)=\frac{1}{\sqrt{1+k(K,B;x,b)^2}}
$$
for $\mathcal{H}^1$ almost all $(x,b)\in \nor(K,B)$. Let $f:\nor(K,B)\to[0,\infty]$ be $\Hc^1$ integrable. Then the coarea formula implies that
\begin{align}
&\int_{\nor(K,B)}f(x,b)\frac{1}{\sqrt{1+k(K,B;x,b)^2}}\, \mathcal{H}^1(\D(x,b))\nonumber\\
&=\int_{\bd K}\int_{\Pi_1^{-1}(\{x\})}f(x,b)\, \mathcal{H}^0(\D(x,b))\, \mathcal{H}^1(\D x).\label{coarea1}
\end{align}
Let $(x,b)\in\nor(K,B)$ be such that $x+(0,\infty)b\subset \DKB$ and $\card\Pi_1^{-1}(\{x\})>1$. Then there is some $\bar b\in\bd B \setminus\{b\}$ such that for all $b'\in\bd B$ from an arc ${\rm arc}_B(b,\bar b)$ connecting $b$ and $\bar b$ we have $x+tb'\in\bd (K+tB)$ for any $t>0$. Let $u:= u(x,b)$ and $v\in u^\perp\cap\, \Sp^1$. Further, let $\gamma:[0,1]\to  {\rm arc}_B(b,\bar b)$ be a map with $\gamma(0)=b$ which is differentiable at $0$ and satisfies $\gamma_r'(0)\not=o$. Since $y:= x+tb$ is a regular boundary point of $K+tB$ and $\Gamma(s):= x+\gamma(s)\in \bd(K+B)$ for $s\in [0,1]$, it follows that $\Gamma'_r(0)= \gamma'_r(0) = \lambda v$ for some $\lambda\not=0$. Clearly, we have $u(K,B,x+tb')=b'$ for all $b'\in{\rm arc}_B(b,\bar b)$, hence $u(K,B,x+t\gamma(s))=\gamma(s)$ for $s\in[0,1]$. The map $u(K,B,\cdot)$ is differentiable at $y$, and the right derivative of $s\mapsto x+t\gamma(s)$ at $s=0$ exists and is equal to $t\lambda v$. This yields
\begin{equation}\label{Dif} 
\dif u(K,B,x+tb)(t\lambda v)=\lambda v
\end{equation}
and thus $\dif u(K,B,y)(v)=t^{-1}v$. Therefore, $k(K,B,x+tb)=1/t$. It follows that $k(K,B;x,b)=\infty$. 
Choosing
$$
f(x,b):=\I\{x\in\bd K:\card \Pi_1^{-1}(\{x\})>1\},
$$
we get from \eqref{coarea1} that $\card \Pi_1^{-1}(\{x\})=1$ for $\mathcal{H}^1$ almost all $x\in\bd K$. Then we apply again \eqref{coarea1} with $f(x,b)=\I\{u(x,b)\in\omega\}$. Since $u(x,b)\in N(K,x)\cap N(B,b)$ by (\ref{Nor}), it follows that $u(x,b)=u_K(x)$, the unique exterior unit normal of $K$ at $x\in\bd K$, for $\mathcal{H}^1$ almost all $x\in\bd K$ and (the unique) $(x,b)\in \nbk$. Then we get
\begin{align*}
&\int_{\nor(K,B)}\I\{u(x,b)\in \omega\}\frac{1}{\sqrt{1+k(K,B;x,b)^2}}\,\mathcal{H}^1(\D(x,b))\\
&=\int_{\bd K}\I\{u_K(x)\in\omega\}\,  \mathcal{H}^1(\D x)=S_1(K,\omega).
\end{align*}

For the proof of the second assertion, we proceed similarly and consider the surjective Lipschitz map $\Pi_2:\nor(K,B)\to\bd K$,  $(x,b)\mapsto b$, with approximate Jacobian
$$
{\rm ap}\, J_1\Pi_2(x,b)=\frac{k(K,B;x,b)}{\sqrt{1+k(K,B;x,b)^2}}
$$
for $\mathcal{H}^1$ almost all $(x,b)\in \nor(K,B)$. If $f:\nor(K,B)\to[0,\infty]$ is measurable, then the
coarea formula implies that
\begin{align}
&\int_{\nor(K,B)}f(x,b)\frac{k(K,B;x,b)}{\sqrt{1+k(K,B;x,b)^2}}\, \mathcal{H}^1(\D(x,b))\nonumber\\
&=\int_{\bd B}\int_{\Pi_2^{-1}(\{b\})}f(x,b)\, \mathcal{H}^0(\D(x,b))\, \mathcal{H}^1(\D b).\label{coarea2}
\end{align}
Let $(x,b)\in\nor(K,B)$ be such that $x+b\in\DKB$ and $\card\Pi_2^{-1}(\{b\})>1$. Then there are $x,\bar x\in\bd K$, $x\neq \bar x$, with $(\bar x,b)\in\nbk$. Then $x+tb\neq \bar x+tb$ are boundary points of $K+tB$ such that for all $x'\in\bd K$ from an arc between $x,\bar x$ we have $x'+tb\in\bd (K+tB)$ and $u(K,B,x'+tb)=b$. Arguing as in the derivation of (\ref{Dif}), we obtain $\dif u(K,B,x+tb)(v)=o =k(K,B,x+tb)v$ and thus $k(K,B;x,b)=0$. Now the proof can be completed as before by applying twice formula \eqref{coarea2}. Here we use again that $u(x,b)\in N(K,x)\cap N(B,b)$ and $u(x,b)=u_B(b)$ for $\mathcal{H}^1$ almost every $b\in\bd B$ and (the unique) $(x,b)\in\nbk$.
\end{proof}

\section{Proof of Theorem \ref{Thm2}}

Let $K\in {\mathcal K}^2$ and $o\in{\rm int}\,K$, without loss of generality. We assume that $K$ and $-K$ are in general relative position and apply the results of Section \ref{relnormal} to $K$ and $B=-K$. Then $(x,y)\in\nor(K,-K)$ if and only if  there is some $u\in{\mathbb S^1}$ such that $x\in F(K,u)$ and $y\in F(-K,u)$. Recall that then $x+y\in F({\rm D}K,u)$ and $-y\in F(K,-u)$. We consider the Lipschitz map
$$
\Phi:\nor(K,-K)\times [0,1]\to K,\quad (x,y,t)\mapsto (1-t)x-ty,
$$
for which we have (recalling that $u(x,y)= u_{K-K}(x+y)$ and $v\in u(x,y)^\perp\cap {\mathbb S}^1$)
\begin{align*}
&\text{\rm ap}J_2\Phi(x,y,t)\\&=\left|\det\left(x+y,\frac{1-t}{\sqrt{1+k(K,-K;x,y)^2}}v-t\frac{k(K,-K;x,y)}{\sqrt{1+k(K,-K;x,y)^2}}v\right)\right|\\
&=h({\rm D}K,u(x,y))\left|\frac{1-t}{\sqrt{1+k(K,-K;x,y)^2}}-\frac{tk(K,-K;x,y)}{\sqrt{1+k(K,-K;x,y)^2}}\right|\\
&\le h({\rm D}K,u(x,y))\left[\frac{1-t}{\sqrt{1+k(K,-K;x,y)^2}}+\frac{tk(K,-K;x,y)}{\sqrt{1+k(K,-K;x,y)^2}}\right] ,
\end{align*}
for $\mathcal{H}^1$-almost all $(x,y)\in\nor(K,-K)$ and all $t\in (0,1)$.  Thus, applying the coarea formula, we get
\begin{align*}
&2V(K)N_a(K)\\
& =\int_K {\rm card}\,\Phi^{-1}(z)\,\D z \\
&=\int_{\nor(K,-K)}\int_0^1\text{\rm ap} J_2\Phi(x,y,t)\,\D t\,\mathcal{H}^1(\D(x,y))\\
&\le \frac{1}{2}\int_{\nor(K,-K)}h(DK,u(x,y))\frac{1}{\sqrt{1+k(K,-K;x,y)^2}}\,\mathcal{H}^1(\D(x,y))\\
&\hspace{11pt}+\frac{1}{2}\int_{\nor(K,-K)}h(DK,u(x,y))\frac{k(K,-K;x,y)}{\sqrt{1+k(K,-K;x,y)^2}}\,\mathcal{H}^1(\D(x,y)).
\end{align*}
An application of Lemma \ref{Lemma4.5} then implies that
\begin{align*}
&2V_2(K)N_a(K)\\
&\le\frac{1}{2}\int_{\Sp^1}h(DK,u)\, S_1(K,\D u) +\frac{1}{2}\int_{\Sp^1}h(DK,u)\, S_1(-K,\D u)\\
&=V(DK,K)+V(DK,-K)\\
&=V_2(DK).
\end{align*}

The right-hand inequality in (\ref{T3}) now follows from the Rogers--Shephard inequality for the difference body (see, e.g., \cite[Section 10.1]{Sch14}) together with the information on the equality sign.

Concerning the left-hand inequality in (\ref{T3}), we remark that any point of a convex body belongs to an affine diameter (\cite{Sol05}, assertion {\bf 3.3}). If $K$ is centrally symmetric and $K$ and $-K$ are in general relative position, which implies that $K$ is strictly convex, then each point of $K$ different from the centre lies on precisely one affine diameter, so equality holds in the left-hand side of (\ref{T3}). Assume next that equality holds there and $K$ and $-K$ are in general relative position. We claim that, under these conditions, all affine diameters of $K$ have a point in common. A theorem of Hammer \cite{H54} (see also Busemann \cite{B55}, pp. 89--90) implies then that $K$ is centrally symmetric.

For the proof, we remark first that every oriented affine diameter $[a,a_1]$ (oriented by demanding that $a_1$ be its endpoint) determines uniquely an angle $\alpha\in[0,2\pi)$ such that $a_1-a$ is a positive multiple of $(\cos\alpha,\sin\alpha)$. We call $\alpha$ the {\em angle of} $[a,a_1]$. Conversely, every $\alpha\in[0,2\pi)$ is the angle of a unique oriented affine diameter $[a,a_1]$. The existence follows from the fact that a longest chord of given direction in a convex body is an affine diameter; see, e.g., \cite{Sol05}, Proposition {\bf 3.1}. The uniqueness follows from the assumption that the boundary of $K$ does not contain segments in distinct parallel support lines. It is then easy to see that an oriented affine diameter depends continuously on its angle.


Assume now, contrary to the claim, that not all affine diameters of $K$ pass through one point. Then, since any two affine diameters intersect, there are three affine diameters $[a,a_1]$, $[b,b_1]$ and $[c,c_1]$ that bound a non-degenerate triangle $\Delta$. We choose the notation so that the points $a,b,c,a_1,b_1,c_1$ come in this order anticlockwise on $\bd K$ (some of the points may coincide), and that $\Delta$ is on the left-hand side of the oriented segments $[a,a_1]$, $[c,c_1]$ and on the right-hand side of $[b,b_1]$. This is clearly possible. We can also choose the coordinate system in such a way that the angles $\alpha,\beta, \gamma$ of the segments $[a,a_1]$, $[b,b_1]$, $[c,c_1]$ satisfy $0=\alpha < \beta < \gamma < \pi$.

Let $x$ be an interior point of $\Delta$. Since $x$ is on the left side of $[a,a_1]$ and on the right side of $[b,b_1]$, there is by continuity an angle between $\alpha$ and $\beta$ such that the oriented affine diameter with this angle passes through $x$. Similarly, there is an angle between $\beta$ and $\gamma$ for which the corresponding oriented affine diameter passes through $x$. The two unoriented affine diameters through $x$ obtained in this way are distinct. We conclude that $N_a(K)\ge 1+\lambda_2 (\Delta)/\lambda_2(K) >1$, a contradiction.

\section{On the Lipschitz continuity of the $B$-projection}

An indispensable prerequisite for the proof of Theorem \ref{Thm2} was Lemma \ref{L4.1}, saying that for convex bodies $K,B\in{\mathcal K}^2$ in general relative position, the $B$-projection to ${\rm bd}\,K$ is a Lipschitz map. We show by an example that there is no corresponding result in higher dimensions. Therefore, the proof of Theorem \ref{Thm2} does not extend to $n\ge 3$. The example might also be of independent interest, as it shows that the metric projection in higher-dimensional Minkowski spaces is in general not Lipschitz continuous.

In order that the $B$-projection be defined, we need an additional assumption. We say that $K,B\in{\mathcal K}^n$ are {\em in strongly general relative position} if
$$ \dim F(K,u) + \dim F(B,u) = \dim F(K+B,u)\qquad \text{for all }u\in{\mathbb S}^{n-1}.$$
In the following, we construct two convex bodies $K,B\in{\mathcal K}^3$ in strongly general relative position for which the $B$-projection to ${\rm bd}\,K$ is not Lipschitz.

In ${\mathbb R}^3$ with the standard basis we consider the points
$$ x_n:= \left(\frac{1}{n},\frac{1}{n^2},0\right),\quad y_n:= \left(\frac{1}{n},\frac{1}{n^2},\frac{1}{n}\right)\quad\text{for } n\in{\mathbb N}.$$

Since the points $x_n$ lie in a plane and on a convex curve and the points $y_n$ lie in a different plane, it is clear that none of the points $x_n,y_n$ lies in the convex hull of the others.

\begin{figure}
\centering
\includegraphics[width=10cm,angle=0]{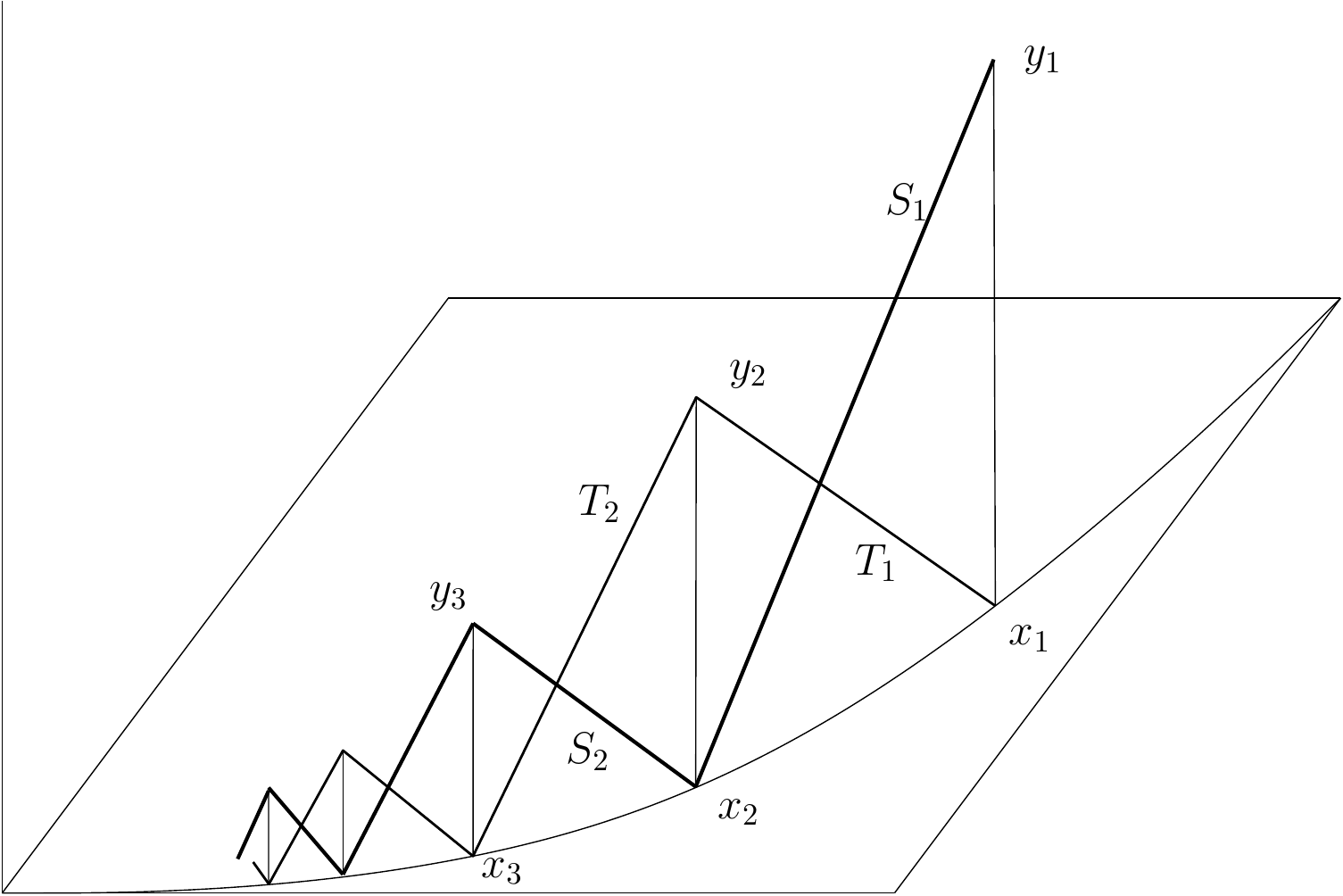}
\caption{The two zig-zag polygons (which are not to scale) are used in the construction of $B$ (heavy lines) and $K$ (heavier lines).}
\end{figure}

For $n\in{\mathbb N}$, we define the segments
$$ S_n: =\left\{\begin{array}{ll} [x_{n+1},y_n] & \text{if $n$ is odd},\\[2mm] [x_n,y_{n+1}] & \text{if $n$ is even},\end{array}\right.
$$
$$
T_n: =\left\{\begin{array}{ll} [x_n,y_{n+1}] & \text{if $n$ is odd},\\[2mm] [x_{n+1},y_n] & \text{if $n$ is even}. \end{array} \right.
$$

Let $n\in{\mathbb N}$. The four points $x_n,y_n, x_{n+1},y_{n+1}$ lie in a plane $H_n$. Let $H_n^0$ be the open halfspace bounded by this plane and containing $o$. Then $x_j,y_j\in H_n^0$ for all $j\notin\{n,n+1\}$. It follows that
$$  H_n\cap {\rm cl}\,{\rm conv} \bigcup_{j\in {\mathbb N}} S_j = S_n,\quad
H_n\cap {\rm cl}\,{\rm conv} \bigcup_{j\in {\mathbb N}} T_j = T_n.$$

Now we define
\begin{eqnarray*}
K &:=& {\rm cl}\,{\rm conv}\left(\bigcup_{n\in {\mathbb N}} S_n \cup \{(0,1,1),(0,1,-1)\}\right),\\
B &:=& {\rm cl}\,{\rm conv}\bigcup_{n\in {\mathbb N}} T_n.
\end{eqnarray*}
It is elementary to check that $K$ and $B$ are in strongly general relative position.

Now let $n\in{\mathbb N}$ be odd. By the properties of the plane $H_n$ mentioned above, there are a unit vector $u$ and a vector $z\in{\mathbb R}^3$ such that
$$ H_n = H(K,u) = H(-B+z,-u).$$
The vector $z_0:= x_n+y_{n+1}$ satisfies
$$ -T_n+z_0=-[x_n,y_{n+1}] +z_0 = [-y_{n+1},-x_n]+z_0=[x_n,y_{n+1}]=T_n,$$
hence $(-T_n+z_0)\cap S_n\not=\emptyset$. There are other vectors $z$ (with $z-z_0$ parallel to $H_n$) for which $-T_n+z$ and $S_n$ intersect in a point $q(z)$. This point is then the unique point in $K\cap (-B+z)$, from which it follows that $ q(z)= p(K,B,z)$. In particular, choosing $z_1:=z_0+x_{n+1}-x_n= x_{n+1}+y_{n+1}$, we get $q(z_1)=x_{n+1}$. Further, we can choose $z_2:= z_0+\lambda(x_n-x_{n+1})$ with suitable $\lambda\in(0,1)$ to obtain $q(z_2)= y_{n+1}+\lambda(x_n-x_{n+1})\in S_n$. With these choices we have
$$ \|p(K,B,z_1)- p(K,B,z_2)\| =\|q(z_1)-q(z_2)\| > \|x_{n+1}-y_{n+1}\|= \frac{1}{n+1}$$
and
$$ \|z_1-z_2\| =(1+\lambda)\|x_n-x_{n+1}\|<\frac{2\sqrt{13}}{(n+1)^2}.$$
Thus we get
$$ \frac{\|p(K,B,z_1)- p(K,B,z_2)\|}{\|z_1-z_2\|}> \frac{n+1}{2\sqrt{13}}.$$
Since here $n$ may be chosen arbitrarily large, this shows that the map $p(K,B,\cdot)$ does not have the Lipschitz property.

To obtain the final counterexample, we may replace $B$ by a translate which has the origin as an interior point. Further, it is not difficult to modify the example in such a way that the gauge body $B$ becomes centrally symmetric with respect to $o$.

\end{document}